\documentclass[10pt]{amsart}
\usepackage[margin=1in]{geometry}
\usepackage{url,amssymb,amsmath,amsthm}
\usepackage[pdftex]{graphicx}
\usepackage[all]{xy}
\usepackage{ stmaryrd }
\usepackage{tikz}
\usetikzlibrary{matrix,arrows}
\title{Remarks on automorphy of residually dihedral representations}
\author{Sudesh Kalyanswamy}
\thanks{Department of Mathematics, UCLA, Los Angeles, USA. \emph{Email}: skalyanswamy@math.ucla.edu}
\date{July 15, 2016}

\DeclareMathOperator{\Ind}{Ind}
\DeclareMathOperator{\ad}{ad}

\DeclareMathOperator{\Frob}{Frob}

\DeclareMathOperator{\GL}{GL}

\DeclareMathOperator{\PGL}{PGL}
\DeclareMathOperator{\CNL}{CNL}

\DeclareMathOperator{\Gal}{Gal}
\DeclareMathOperator{\tr}{tr}

\DeclareMathOperator{\St}{St}

\newcommand{\cO}{{\mathcal O}}

\newcommand{\A}{\mathbb{A}}

\newcommand{\C}{\mathbb{C}}

\newcommand{\Z}{\mathbb{Z}}

\newcommand{\Q}[0]
{
\mathbb{Q}
}

\def\rhobar{ {\overline{\rho}} }

\newcommand{\F}[0]
{
\mathbb{F}
}

\begin{document}
\begin{abstract} We prove automorphy lifting results for  geometric representations $\rho:G_F \rightarrow GL_2(\cO)$, with $F$ a totally real field,  and $\cO$ the ring of integers of a finite extension of $\Q_p$ with  $p$ an odd prime, such that the residual representation $\rhobar$ is  totally odd and induced from a  character of the  absolute Galois group  of the quadratic subfield $K$ of $F(\zeta_p)/F$.
Such representations fail the Taylor-Wiles hypothesis and the patching techniques to prove automorphy  do not work.  We apply this to  automorphy of elliptic curves $E$ over $F$ , when $E$ has no $F$ rational 7-isogeny and such that the image of $G_F$ acting on $E[7]$ normalizes a split Cartan subgroup of $\GL_2(\F_7)$.

 \end{abstract}
\maketitle
\newcommand{\twopartdef}[4]
{
	\left\{
		\begin{array}{ll}
			#1 & \mbox{if } #2 \\
			#3 & \mbox{if } #4
		\end{array}
	\right.
}
\newcommand{\I}[1]
{
\mathfrak{#1}
}

\newcommand{\D}[1]
{
#1^{\vee}
}
\newcommand{\U}[1]
{
	{#1}^{\times}
}

\newcommand{\CharFun}[1]
{
\textbf{1}_{#1}
}

\newcommand{\Al}[0]
{
\mathcal{O}
}
\newcommand{\Mod}[0]
{
\text{ mod }
}
\newcommand{\Minus}[0]
{
\backslash
}
\newcommand{\AC}[1]
{
	\overline{#1}
}
\newcommand{\MatTwo}[4]
{
\left(
\begin{array}{ccc}
#1 & #2 \\
#3 & #4 \\
\end{array}
\right)
}
\newcommand{\MatThree}[9]
{
\left(
\begin{array}{ccc}
#1 & #2 & #3 \\
#4 & #5 & #6 \\
#7 & #8 & #9 \\
\end{array}
\right)
}
\newcommand{\invlim}[1]
{
\lim_{\stackrel{\longleftarrow}{#1}}
}
\newcommand{\Sc}[1]
{
	\mathcal{#1}
	}
\newcommand{\IP}[1]
{
	\left\langle #1 \right\rangle
	}
\theoremstyle{definition}
\newtheorem{theorem}{Theorem}[section]
\newtheorem{lemma}[theorem]{Lemma}
\newtheorem{corollary}[theorem]{Corollary}
\newtheorem{proposition}[theorem]{Proposition}
\newtheorem{definition}[theorem]{Definition}
\newtheorem{example}[theorem]{Example}
\theoremstyle{remark}
\newtheorem{remark}[theorem]{Remark}
\newtheorem*{claim}{Claim}
\newtheorem*{proofofclaim}{Proof of Claim}
\setlength{\parindent}{0cm}

\section{Introduction}

\subsection{The main theorem}

Let $F$ be a totally real field, $p$ be an odd prime, and $\cO$ the ring of integers of a finite extension of $\Q_p$. In proving automorphy of geometric representations $\rho:G_F \rightarrow GL_2(\cO)$ which are residually automorphic, there is an assumption made (the Taylor-Wiles hypothesis)  that the residual representation $\rhobar$ is irreducible when restricted  to $F(\zeta_p)$. In \cite{Thorne}, Thorne has recently weakened this assumption in many cases to simply asking that $\rhobar$ itself is irreducible but making the auxiliary assumption that the quadratic subfield of $F(\zeta_p)/F$ is totally real. \\ \\
 Under other assumptions on $\rhobar$, but still allowing that $\rhobar|_{G_{F(\zeta_p)}}$ is reducible, we prove automorphy lifting results wherein we do not assume that the quadratic subfield $K$ of $F(\zeta_p)/F$ is totally real. The assumption on $\rhobar$ is that there is a (``level raising'') place $v$ of $F$ that splits in $K$, such that the ratio of the eigenvalues of $\rhobar({\rm Frob}_v)$ is $q_v$ with $q_v$ not 1 modulo $p$. This assumption is automatic when $K/F$ is totally real, as  exploiting the oddness of $\rhobar$ we can take any $v$ such that $\rhobar({\rm Frob}_v)$ is conjugate to the image of complex conjugation. \\ 

The main theorem we  prove is the following:
\begin{theorem} Let $F$ be a totally real number field, let $p$ be an odd prime, and let $\rho: G_F \to \GL_2(\overline{\Q}_p)$ be a continuous representation satisfying the following:
\begin{itemize}
\item[(1)] The representation $\rho$ is almost everywhere unramified. 
\item[(2)] For each $v|p$ of $F$, the local representation $\rho|_{G_{F_v}}$ is de Rham. For each embedding $\tau: F \hookrightarrow \overline{\Q}_p$, we have $\text{HT}_{\tau}(\rho) = \{0,1\}$. 
\item[(3)] For each complex conjugation $c \in G_F$, we have $\det \rho(c) = -1$. 
\item[(4)] The residual representation $\bar{\rho}$ is absolutely irreducible, but $\bar{\rho}|_{G_{F(\zeta_p)}}$ is a direct sum of two distinct characters. Further suppose that if $K$ is the unique quadratic subfield of $F(\zeta_p)/F$ and $\bar{\gamma}: G_K \to k^{\times}$ is the ratio of the two characters, then we have $F(\zeta_p) \not \subset K(\bar{\gamma}\epsilon^{-1}) \cap K(\bar{\gamma}\epsilon)$. 
\end{itemize}
Then $\rho$ is automorphic: there exists a cuspidal automorphic representation $\pi$ of $\GL_2(\A_F)$ of weight $2$, an isomorphism $\iota: \overline{\Q}_p \to \C$, and an isomorphism $\rho \cong r_{\iota}(\pi)$. 
\end{theorem} 

The hypothesis that $F(\zeta_p) \not \subset K(\bar{\gamma}\epsilon^{-1}) \cap K(\bar{\gamma}\epsilon)$ is equivalent to the existence of a place $v$ of $F$ such that $q_v$ is not 1 modulo $p$, $v$ splits in $K$, and $\rhobar({\rm Frob}_v)$ has eigenvalues with ratio $q_v$. Furthermore, as remarked above, this hypothesis is automatic when $K$ is totally real as one may, using the oddness of $\rhobar$, take any $v$ such that $\rhobar({\rm Frob}_v)$ is conjugate to $\rhobar(c)$ and $q_v$ is -1 modulo $p$. Thus the  theorem generalizes the main theorem of \cite{Thorne}, and in fact  its  proof is essentially  the same as \cite{Thorne}. In loc. cit., the author synthesizes two methods:

--  the patching method of Taylor-Wiles, which relies on using places $v$ of $F$ such that $q_v$ is $1$ modulo $p^N$ and $\rhobar({\rm Frob}_v)$ has distinct eigenvalues to kill certain elements of the mod $p$ dual Selmer, and

-- a method due to Khare, which proves automorphy of $\rho$ by ``$p$-adic approximation'' using Ramakrishna places $v$, such that $q_v$ is not $1$ modulo $p$ and $\rho_N({\rm Frob}_v)$ has eigenvalues with ratio $q_v$, to perform the task of killing certain elements of mod $p$ dual Selmer which cannot be killed with Taylor-Wiles primes. Here $\rho_N$ denotes the mod $p^N$ reduction of $\rho$. \\ 

The Ramakrishna places $v$ that Thorne used were such that $\rho_N({\rm Frob}_v)$ is conjugate to the image of complex conjugation. To ensure that these places could effectively kill the troublesome part of mod $p$ dual Selmer, he had to impose that the quadratic subfield of $F(\zeta_p)/F$ be totally real, which implies $p \equiv 1 \Mod 4$. \\

The central new remark of this paper is that if we assume there is a place $v$ of $F$ that splits in $K$, such that the ratio of the eigenvalues of $\rhobar({\rm Frob}_v)$ is $q_v$ with $q_v$ not $1$ modulo $p$, then by choosing ``Teichmuller liftings'' one may find, for each $N$, places $v$ such that $q_v$ is not $1$ modulo $p$, $\rho_N({\rm Frob}_v)$ has eigenvalues with ratio $q_v$, and the $v$'s perform the task of killing the troublesome  elements of mod $p$ dual Selmer which cannot be killed with Taylor-Wiles primes. We present the proof of this remark in Proposition 2.11, and it serves as a replacement for Proposition 5.20 in \cite{Thorne} which needed that $K$ be totally real. Note that this assumption we make is automatic when $K$ is totally real, which explains why our result generalizes the main theorem of loc. cit.  \\ 

After this remark the proof of our theorem is identical to that of loc. cit. where patching arguments followed by level raising and lowing arguments are used to conclude the automorphy of $\rho$. \\ 

The case when there is no place $v$ of $F$ that splits in $K$ such that the ratio of the eigenvalues of $\rhobar({\rm Frob}_v)$ is $q_v$ with $q_v$ not 1 modulo $p$ cannot be addressed by the methods of \cite{Thorne} nor by the modifications that we carry out in this paper. In a forthcoming work \cite{KT}  of Khare and Thorne, this case is addressed. Instead of patching and  $p$-adic approximation, the authors use patching together with Wiles' numerical isomorphism criterion and calculation of $\eta$-invariants using monodromy arguments.

\subsection{Automorphy of elliptic curves}
 
We get a modest application to automorphy of elliptic curves over totally real fields.
Current automorphy lifting theorems prove the modularity of elliptic curves $E$ over $F$ which satisfy the ``Taylor-Wiles'' hypothesis at some prime $p \in \{3,5,7\}$, i.e.,  $\bar{\rho}_{E,p}(G_{F(\zeta_p)})$ is absolutely irreducible  for at least one  $p \in \{3,5,7\}$. Recently, in \cite{FLHS}, the authors prove the modularity of all elliptic curves over real quadratic fields. Moreover, they show that over any totally real field $F$, there are only finitely many potentially non-modular elliptic curves over $F$, which necessarily do not satisfy the Taylor-Wiles hypothesis at any prime $p \in \{3,5,7\}$. \\ \\
In \cite{Thorne}, Thorne was able to cut down on this list of potentially non-modular elliptic curves by proving a new automorphy lifting theorem and applying it to the above situation of elliptic curves with $p=5$. However, the results in that paper do not deal with the case $p=7$ because the requirement was that the quadratic subfield of $F(\zeta_p)/F$ be totally real (which implies that $p$ is 1 modulo 4). In this note, we try to deal with the cases when this quadratic subfield is not totally real by examining different sorts of auxiliary primes. We use this theorem to prove:
\begin{theorem} Let $F$ be a totally real field, and let $E$ be an elliptic curve over $F$. Suppose:
\begin{itemize}
\item[(1)] $F \cap \Q(\zeta_7) = \Q$. 
\item[(2)] $E$ has no $F$-rational $7$-isogeny.
\item[(3)] Either $\bar{\rho}_{E,7}(G_{F(\zeta_7)})$ is absolutely irreducible, or it is reducible and is conjugate to a subgroup of a split Cartan subgroup of $\GL_2(\F_7)$. 
\end{itemize}
Then $E$ is modular. 
\end{theorem}

We further note that one may also deduce by similar arguments   the automorphy of elliptic curves $E$ defined over a  totally real  field $F$,  with a prime $p$ such that $E$ has no $F$-rational $p$-isogeny,  $\rho_{E,p}(G_F)$ normalizes a split Cartan subgroup of $\GL_2(\F_p)$,   and $E$ has bad semistable reduction at a place $v$ such that
$q_v$ is not 1 mod $p$ and $v$ splits in the quadratic subfield of $F(\zeta_p)/F$.

\subsection{Acknowledgements} 
I want to thank Chandrashekhar Khare for many useful discussions and his invaluable guidance regarding the methods used in this paper.

\section{Galois Theory}
\subsection{Global Deformation Problem}
Let $F$ be a totally real number field. Let $p$ be an odd prime, and let $E$ be a coefficient field. Fix a continuous, absolutely irreducible representation $\bar{\rho} \colon G_F \to \GL_2(k)$ and a continuous character $\mu: G_F \to \Sc{O}^{\times}$ lifting $\det \bar{\rho}$. We assume $k$ contains the eigenvalues of all elements in the image of $\bar{\rho}$.\\ \\ Let $S$ be a finite set of finite places of $F$, containing the places dividing $p$ and the places at which $\bar{\rho}$ and $\mu$ are ramified. For $v \in S$, fix a ring $\Lambda_v \in \CNL_{\Sc{O}}$. Define $\Lambda = \hat{\otimes}_{v \in S} \Lambda_v$, so that $\Lambda \in \CNL_{\Sc{O}}$. \\ \\ If $v \in S$, we write $\Sc{D}_v^{\Box} \colon \CNL_{\Lambda_v} \to \text{Sets}$ for the functor which takes $R \in \CNL_{\Lambda_v}$ to the set of continuous homomorphisms $r \colon G_{F_v} \to \GL_2(R)$ such that $r \Mod \I{m}_R = \bar{\rho}|_{G_{F_v}}$ and $\det r$ agrees with the composite
$$G_{F_v} \stackrel{\mu|_{G_{F_v}}}{\longrightarrow} \Sc{O}^{\times} \longrightarrow R^{\times}.$$ This functor $\Sc{D}_v^{\Box}$ is represented by an object $R_v^{\Box} \in \CNL_{\Lambda_v}$. 

\begin{definition} Let $v \in S$. A local deformation problem for $\bar{\rho}|_{G_{F_v}}$ is a subfunctor $\Sc{D}_v \subset \Sc{D}_v^{\Box}$ satisfying:
\begin{itemize}
\item[(i)] The subfunctor $\Sc{D}_v$ is represented by a quotient $R_v$ of $R_v^{\Box}$
\item[(ii)] For all $R \in \CNL_{\Lambda_v}$, $a \in \ker(\GL_2(R) \to \GL_2(k))$ and $r \in \Sc{D}_v(R)$, we have $ara^{-1} \in \Sc{D}_v(R)$. 
\end{itemize}
\end{definition}

\begin{definition} A global deformation problem is a tuple
$$\Sc{S} = (\bar{\rho},\mu,S,\{\Lambda_v\}_{v \in S},\{\Sc{D}_v\}_{v \in S}),$$ where all the notation is as above, and $\Sc{D}_v$ is a local deformation problem for $\bar{\rho}|_{G_{F_v}}$ for $v \in S$. 
\end{definition}
\subsection{Galois Cohomology}
Given a finite place $v$ of $F$ and a local deformation problem $\Sc{D}_v$, we denote by $\Sc{L}_v$ the tangent space of the local deformation problem as a subspace $\Sc{L}_v \subset H^1(F_v,\ad^0 \bar{\rho})$. We also let $\Sc{L}_v^{\perp} \subset H^1(F_v,\ad^0 \bar{\rho}(1))$ be the annihilator of $\Sc{L}_v$ induced by the perfect pairing of Galois modules:
$$\ad^0 \bar{\rho} \times \ad^0 \bar{\rho}(1) \to k(1), \quad (X,Y) \mapsto \tr(XY).$$ 
\begin{definition} Given a global deformation problem $\Sc{S} = (\bar{\rho},\mu,S,\{\Lambda_v\}_{v \in S},\{\Sc{D}_v\}_{v \in S})$, we define the dual Selmer group
$$H^1_{\Sc{S},T}(\ad^0 \bar{\rho}(1)) = \ker \left( H^1(F_S/F,\ad^0 \bar{\rho}(1)) \to \prod_{v \in S-T} H^1(F_v,\ad^0 \bar{\rho}(1))/\Sc{L}_v^{\perp} \right).$$
\end{definition}

Let us now take a look at the specific deformation problem we want to consider.
\subsection{Special Deformations, case $q_v \not \equiv 1 \Mod p$}
Let $v \in S$ be a prime not dividing $p$, and suppose that $q_v \not \equiv 1 \Mod p$. Suppose further that $\bar{\rho}|_{G_{F_v}}$ is unramified, and that $\bar{\rho}(\Frob_v)$ has two distinct eigenvalues $\alpha_v ,\beta_v \in k$ such that $\alpha_v/\beta_v = q_v$. Let $\Lambda_v = \Sc{O}$. We define a subfunctor $\Sc{D}_v^{\St(\alpha_v)} \subset \Sc{D}_v^{\Box}$ directly. Let $R \in \CNL_{\Sc{O}}$ and let $r: G_{F_v} \to \GL_2(R)$ be an element of $\Sc{D}_v^{\Box}$. If $\phi_v \in G_{F_{v}}$ is a choice of geometric Frobenius, then by Hensel's lemma the characteristic polynomial of $r(\phi_v)$ factors as $(X-A_v)(X-B_v)$, where $A_v,B_v \in R^{\times}$ with $A_v$ lifting $\alpha_v$ and $B_v$ lifting $\beta_v$. We will say $r \in \Sc{D}_v^{\St(\alpha_v)}$ if $A_v = q_vB_v$ and $I_{F_v}$ acts trivially on $(r(\phi_v)-A_v)R^2$, which is a direct summand $R$-submodule of $R^2$. One checks that this condition is independent of the choice of $\phi_v$. 
\begin{proposition} The functor $\Sc{D}_v^{\St(\alpha_v)}$ is a local deformation problem. The representing object $R_v^{\St(\alpha_v)}$ is formally smooth over $\Sc{O}$ of dimension $4$. 
\end{proposition}

\subsection{Existence of Auxiliary Primes}
Continue with the notation from the previous section, and assume further that $\bar{\rho}$ is totally odd, i.e. that $\mu(c) = -1$ for all choices of complex conjugation $c \in G_F$. Write $\zeta_p \in \overline{F}$ for a primitive $p$-th root of unity, and now fix a choice of complex conjugation $c \in G_F$. \\ \\We will assume that $\bar{\rho}|_{G_{F(\zeta_p)}}$ is the direct sum of two distinct characters. By Clifford theory, we know that $\bar{\rho}$ is induced from a continuous character $\bar{\chi}: G_K \to k^{\times}$, where $K$ is the unique quadratic subfield of $F(\zeta_p)/F$. That is,
$\bar{\rho} \cong \Ind_{G_K}^{G_F} \bar{\chi}$. Write $w \in G_F$ for a fixed choice of element with nontrivial image in $\Gal(K/F)$. We can assume that, possibly after conjugation, that $\bar{\rho}$ has the form:
$$\bar{\rho}(\sigma) = \MatTwo{\bar{\chi}(\sigma)}{0}{0}{\bar{\chi}^w (\sigma)}, \quad \text{ for } \sigma \in G_K,$$
$$\bar{\rho}(w) = \MatTwo{0}{\bar{\chi}(w^2)}{1}{0}.$$

Now let $\bar{\gamma} = \bar{\chi}/\bar{\chi}^w$. By assumption, $\bar{\gamma}$ is nontrivial, even after restriction to $G_{F(\zeta_p)}$. We have the following:
\begin{lemma} We have that $\ad^0 \bar{\rho}$ decomposes as $\ad^0 \bar{\rho} \cong k(\delta_{K/F}) \oplus \Ind_{G_K}^{G_F} \bar{\gamma}$ as a $G_F$-module, where $\delta_{K/F}: \Gal(K/F) \to k^{\times}$ is the unique nontrivial character. 
\end{lemma}
From now on, we will let $M_0 = k(\delta_{K/F})$ and $M_1 = \Ind_{G_K}^{G_F} \bar{\gamma}$. Fix the standard basis for $\ad^0 \bar{\rho}$:
$$E = \MatTwo{0}{1}{0}{0},\quad H = \MatTwo{1}{0}{0}{-1},\quad F = \MatTwo{0}{0}{1}{0}.$$ If $M \in \ad^0 \bar{\rho}$, we write $k_M \subset \ad^0 \bar{\rho}$ for the line that it spans. \\ \\ We have the following easy lemma:
\begin{lemma} Let $v \nmid p$ be a finite place of $F$ which splits in $K$, and suppose that the local deformation problem $\Sc{D}_v = \Sc{D}_v^{\St(\alpha_v)}$ is defined. 
\begin{itemize}
\item[(1)] The subspace $\Sc{L}_v \subset H^1(F_v,\ad^0 \bar{\rho})$ respects the decomposition $\ad^0 \bar{\rho} = M_0 \oplus M_1$. That is,
$$\Sc{L}_v = (\Sc{L}_v \cap H^1(F_v,M_0)) \oplus (\Sc{L}_v \cap H^1(F_v,M_1)).$$
\item[(2)] The subspace $\Sc{L}_v^{\perp} \subset H^1(F_v,\ad^0 \bar{\rho}(1))$ respects the decomposition $\ad^0 \bar{\rho}(1) = M_0(1) \oplus M_1(1)$. 
\end{itemize}
\end{lemma}
\begin{proof} The second part is dual to the first, so we only prove the first part. The fact that $\Sc{D}_v^{\St(\alpha_v)}$ is defined means $q_v \not \equiv 1 \Mod p$, that $\bar{\rho}|_{G_{F_v}}$ is unramified, and that $\bar{\rho}(\Frob_v)$ takes two distinct eigenvalues $\alpha_v,\beta_v \in k$ with $\alpha_v/\beta_v = q_v$. The fact that $v$ splits in $K$ means $M_0 = k_H$ and $M_1 = k_E(\epsilon) \oplus k_F(\epsilon^{-1})$ as $k[G_{F_v}]$-modules. The case $q_v \equiv -1 \Mod p$ was proved in \cite{Thorne}. If $q_v \not \equiv  \pm 1 \Mod p$, then $\Sc{L}_v$ is $1$-dimensional, and is contained in $H^1(F_v,M_1)$, being spanned by $H^1(F_v,k_E(\epsilon))$. 
\end{proof}
\begin{remark} The difference between this and the corresponding Lemma 5.18 from \cite{Thorne} is that we do not make the assumption that the inducing field $K$ is totally real. However, we do need to make sure that we choose primes of $F$ which split in $K$ for the rest of the method to work. 
\end{remark}
Let $\Sc{S} = (\bar{\rho},\mu,S,\{\Lambda_v\}_{v \in S},\{\Sc{D}_v\}_{v \in S})$ be a global deformation problem, and let $T \subset S$ be a subset containing $S_p$ (the set of places above $p$). Suppose that for $v \in S-T$, the local deformation problem $\Sc{D}_v = \Sc{D}_v^{\St(\alpha_v)}$. The lemma implies we can decompose
$$H^1_{\Sc{S},T}(\ad^0 \bar{\rho}(1)) = H^1_{\Sc{S},T}(M_0(1)) \oplus H^1_{\Sc{S},T}(M_1(1)).$$
We now show that we can kill the $M_1(1)$ portion of dual Selmer using the special deformation problem we defined in the previous section, and then kill the $M_0(1)$ portion using traditional Taylor-Wiles primes. 

\subsection{Killing the $M_1(1)$ portion}
First, we show that lemmas 5.21, 5.22, and 5.23 of \cite{Thorne} continue to hold even if $K$ is not totally real. Indeed, the only one which requires proof is the second, since this is the only place where Thorne used this assumption. However, we will need to impose an additional restriction. We state the other two lemmas here for convenience. 
\begin{lemma} Let $\Gamma$ be a group, and $\alpha: \Gamma \to k^{\times}$ a character. Let $k' \subset k$ be the subfield generated by the values of $\alpha$. Then $k'(\alpha)$ is a simple $\F_p[\Gamma]$-module. If $\beta: \Gamma \to k^{\times}$ is another character, then $k'(\alpha)$ is isomorphic to a $\F_p[\Gamma]$-submodule of $k(\beta)$ if and only if there is an automorphism $\tau$ of $k$ such that $\beta = \tau \circ \alpha$. 
\end{lemma}
\begin{proof} See the proof of Lemma 5.21 in \cite{Thorne}.
\end{proof}
\begin{lemma} Let $K(\overline{\gamma})$ be the fixed field of $\ker \bar{\gamma}$, let $L = F(\zeta_p) \cap K(\overline{\gamma})$ and assume that $\# \epsilon(G_L )> 1$. Then the $\F_p[G_K]$-module $k(\epsilon \overline{\gamma})$ has no Jordan-Holder factors in common with $k$, $k(\overline{\gamma})$, or $k(\overline{\gamma}^{-1})$. The characters $\epsilon \overline{\gamma}$ and $\overline{\gamma}$ are nontrivial. 
\end{lemma} 
\begin{proof} The second claim follows from the fact that $\overline{\gamma}|_{G_{F(\zeta_p)}}$ is nontrivial. For the first claim, we show there are no $\F_p[G_K]$-module homomorphisms from $k(\epsilon \overline{\gamma})$ to $k(\overline{\gamma})$ or $k(\overline{\gamma}^{-1})$. Let $f: k(\epsilon \overline{\gamma}) \to k(\overline{\gamma})$ be such a homomorphism, choose $a \in k(\epsilon \overline{\gamma})$, and assume $f(a) = b$ . By the hypothesis of the lemma, there is an element $\tau \in G_L \subset G_K$ such that $\epsilon(\tau) \neq 1$ and $\gamma(\tau) = 1$. Since $f$ is a $\F_p[G_K]$-module homomorphism and $\epsilon(\tau) \in \F_p^{\times}$, we know
$$f(\epsilon(\tau)\overline{\gamma}(\tau)a) = \epsilon(\tau)f(a).$$ On the other hand,
$$f(\epsilon(\tau)\overline{\gamma}(\tau)a) = \overline{\gamma}(\tau)b = b.$$ Thus, $\epsilon(\tau)b = b$, which implies $b = 0$. Since $a \in k(\epsilon \overline{\gamma})$ was arbitrary, this implies $f = 0$. Thus, there are no nontrivial homomorphisms between $k(\epsilon \overline{\gamma})$ and $k(\overline{\gamma})$. \\ \\ The same proof shows there are no nontrivial homomorphisms between $k(\epsilon \overline{\gamma})$ and $k(\overline{\gamma}^{-1})$ or $k$. 
\end{proof}
\begin{lemma} Let $N \geq 1$ and let $K_N = F(\zeta_{p^N},\overline{\rho}_N)$, i.e. $K_N$ is the splitting field of $\overline{\rho}_N|_{F(\zeta_{p^N})}$. Then $H^1(K_N/F,M_1(1)) = 0$. 
\end{lemma}
\begin{proof} When $K$ is totally real, this is proved in Lemma 5.23 of \cite{Thorne}. The same proof proves the lemma in the case $K$ is CM using the preceding lemma. 
\end{proof}
The following proposition is the analog of  Proposition 5.20 of \cite{Thorne} and is  the only place where we argue differently from Thorne  because of not having  (in the case that $K$ is not totally real) the luxury to choose places $v$ such that $\rho_N({\rm Frob}_v)$ is the image of complex conjugation under $\rho_N$. The proof relies on the simple observation  that given an element $g$ in $\GL_2(\cO/p^M)$, then for $N>>0$, the element $g^{q^N}$ has (ratio of) eigenvalues that are the Teichm\"uller lift of the (ratio of) the eigenvalues of the reduction of $g$.
\begin{proposition} Let $\Sc{S} = (\bar{\rho},\mu,S,\{\Lambda_v\}_{v \in S},\{\Sc{D}_v\}_{v \in S})$ be a global deformation problem, and let $T = S$. Let $N_0 \geq 1$ be an integer. Let $\rho: G_F \to \GL_2(\Sc{O})$ be a lifting of type $\Sc{S}$. Let $K(\overline{\gamma}\epsilon)$ (resp. $K(\overline{\gamma}\epsilon^{-1})$) be the fixed field of $\ker \bar{\gamma}\epsilon$ (resp. $\ker \bar{\gamma} \epsilon^{-1}$), and assume that $F(\zeta_p) \not \subset K(\overline{\gamma}\epsilon) \cap K(\overline{\gamma}\epsilon^{-1})$. Then for any $m \geq h^1_{\Sc{S},T}(M_1(1))$, there exists a set $Q_0$ of primes, disjoint from $S$, and elements $\alpha_v \in k^{\times}$, satisfying the following:
\begin{itemize}
\item[(1)] $\# Q_0 = m$
\item[(2)] For each $v \in Q_0$, the local deformation problem $\Sc{D}_v^{\St(\alpha_v)}$ is defined. We define the augmented deformation problem
$$\Sc{S}_{Q_0} = (\bar{\rho},\mu,S \cup Q_0,\{\Lambda_v\}_{v \in S} \cup \{\Sc{O}\}_{v \in Q_0}, \{\Sc{D}_v\}_{v \in S} \cup \{\Sc{D}_v^{\St(\alpha_v)}\}_{v \in Q_0}).$$
\item[(3)] Let $\rho_{N_0} = \rho \Mod \lambda^{N_0}: G_F \to \GL_2(\Sc{O}/\lambda^{N_0}\Sc{O})$. Then $\rho_{N_0}(\Frob_v)$ has distinct eigenvalues whose ratio is $q_v$ for each $v \in Q_0$. 
\item[(4)] $H^1_{\Sc{S}_{Q_0},T}(M_1(1)) = 0$. 
\end{itemize}
\end{proposition}

\begin{proof} We wish to find a set $Q_0$ of primes such that $h^1_{\Sc{S}_{Q_0},T}(M_1(1)) = 0$. Suppose $r = h^1_{\Sc{S},T}(M_1(1)) \geq 0$. Using induction, it suffices to find a single prime $v$ satisfying the conditions of the theorem such that $h^1_{\Sc{S}_{\{v\}},T}(M_1(1)) = \max(r-1,0)$. The case $r=0$ is easy, so assume $r \geq 1$. \\  

Let $0 \neq [\varphi] \in H^1_{\Sc{S},T}(M_1(1))$ be a nonzero class. We wish to find a place $v \notin S$ such that:
\begin{itemize}
\item[(i)] $v$ splits in $K$
\item[(ii)] $\rho_{N_0}(\Frob_v)$ has distinct eigenvalues with ratio $q_v \Mod \lambda^{N_0}$. 
\item[(iii)] $q_v \not  \equiv \pm 1 \Mod \lambda^{N_0}$
\item[(iv)] $\varphi(\Frob_v) \neq 0$ ($\in M_1(1)$). 
\end{itemize}
Indeed, the first three conditions imply that $\Sc{D}_v^{\St(\alpha_v)}$ is defined for the appropriate choice of $\alpha_v$. We also have an exact sequence
$$0 \to H^1_{\Sc{S}_{\{v\}},T}(M_1(1)) \to H^1_{\Sc{S},T}(M_1(1)) \to H^1(G_{F_v},M_1(1)) \cong k.$$  Condition (iv) implies the final map is surjective, which gives $h^1_{\Sc{S}_v,T}(M_1(1)) < h^1_{\Sc{S},T}(M_1(1))$, as desired. \\ \\ By the Cebotarev density theorem, it suffices to find an element $\sigma \in G_K$ such that:
\begin{itemize}
\item[(a)] $\rho_{N_0}(\sigma)$ has distinct eigenvalues with ratio $\epsilon(\sigma) \Mod \lambda^{N_0}$. 
\item[(b)] $\epsilon(\sigma) \not \equiv \pm 1 \Mod \lambda^{N_0}$
\item[(c)] $\varphi(\sigma) \neq 0$. 
\end{itemize}
If $N_0 = 1$, then the assumption in the theorem ensures we can find $\sigma_1$ in $G_K$ such that $\overline{\gamma}(\sigma_1) = \epsilon(\sigma_1) \Mod \lambda$. Indeed, the assumption $F(\zeta_p) \not \subset K(\overline{\gamma}\epsilon) \cap K(\overline{\gamma}\epsilon^{-1})$ ensures that either $G_{K(\overline{\gamma}\epsilon)}$ or $G_{K(\overline{\gamma}\epsilon^{-1})}$ is not contained in $G_{F(\zeta_p)}$. This means there exists $\sigma_1$ in either  $G_{K(\overline{\gamma}\epsilon)}$ or $G_{K(\overline{\gamma}\epsilon^{-1})}$ such that $\epsilon(\sigma_1) \not \equiv 1 \Mod p$. In the latter case, we find our desired $\sigma_1$. In the former case, by exchanging the roles of the eigenvalues, we get our desired $\sigma_1$.\\ \\ If $\varphi(\sigma_1) \neq 0$, then take $\sigma = \sigma_1$, so suppose $\varphi(\sigma_1) = 0$. We have the inflation-restriction sequence:
$$0 \to H^1(K_1/F, M_1(1)^{G_{K_1}}) \to H^1(F,M_1(1)) \to H^1(K_1,M_1(1))^{\Gal(K_1/F)}.$$ By the previous lemma, the first group is zero, so the image of $\varphi$ in $H^1(K_1,M_1(1))$ is nonzero. This restriction is a nonzero homomorphism $\varphi|_{G_{K_1}}: G_{K_1} \to M_1(1)$. Thus, we can find $\tau \in G_{K_1}$ such that $\varphi(\tau) \neq 0$. Then take $\sigma = \tau \sigma_1$. Then $$\overline{\rho}(\sigma) = \overline{\rho}(\tau) \overline{\rho}(\sigma_1) = \overline{\rho}(\sigma_1)$$ as $\tau \in \ker(\overline{\rho})$. We also find 
$$\epsilon(\sigma) = \epsilon(\tau) \epsilon(\sigma_1)  \equiv \epsilon(\sigma_1) \Mod \lambda$$ as $\epsilon(\tau) \equiv 1 \Mod \lambda$. Thus, $\overline{\gamma}(\sigma) \equiv \epsilon(\sigma) \Mod \lambda$. Moreover, 
$$\varphi(\sigma) = \varphi(\tau) + \varphi(\sigma_1),$$ meaning $\varphi(\sigma) \neq 0$, as required. \\ \\ If $N_0 > 1$, then consider the element $\sigma_1$ defined above. Then $\rho_{N_0}(\sigma_1)$ has distinct eigenvalues by Hensel's lemma, and we know this ratio modulo $\lambda$ is $\overline{\gamma}(\sigma_1) \equiv \epsilon(\sigma_1) \Mod \lambda$. Consider $\sigma_{N_0} = \sigma_1^{q^M}$ for some $M$ to be determined and $q = \# k$. For some sufficiently high power of $M$, $\epsilon(\sigma_{N_0}) \Mod \lambda^{N_0}$ will be the Teichmuller lifting of $\epsilon(\sigma_1) \Mod \lambda$ to the $\Mod \lambda^{N_0}$ ring (indeed, $M = q^{N_0-1}$ should do). But since $\overline{\gamma}(\sigma_1) \equiv \epsilon(\sigma_1) \Mod \lambda$, we deduce that the ratio of the eigenvalues of $\rho_{N_0}(\sigma_{N_0})$ will be equivalent to $\epsilon(\sigma_{N_0}) \Mod \lambda^{N_0}$. \\ \\ We still need to make sure $\varphi(\sigma) \neq 0$. If $\varphi(\sigma_{N_0}) \neq 0$, then we can take $\sigma = \sigma_{N_0}$. If $\varphi(\sigma_{N_0}) = 0$, consider $\tau \in G_{K_N}$ with $\varphi(\tau) \neq 0$ as before. Let $\sigma = \tau \sigma_{N_0}$. By the same reasoning as in the $N_0=1$ case, the ratio of the eigenvalue of $\rho_N(\sigma)$ will still be equivalent to $\epsilon(\sigma) \Mod \lambda^{N_0}$, and moreover $\varphi(\sigma) = \varphi(\tau) + \varphi(\sigma_{N_0}) \neq 0$ by construction. This concludes the proof.
\end{proof}

\begin{remark} Note that the assumption   $F(\zeta_p) \not \subset K(\overline{\gamma}\epsilon) \cap K(\overline{\gamma}\epsilon^{-1})$ is implied by the more checkable condition that $(\# \epsilon(G_L ), \# \overline{\gamma}(G_L)) > 1$, where $L = F(\zeta_p) \cap K(\overline{\gamma})$. 
\end{remark}

\subsection{Killing the $M_0(1)$ portion}
Having killed the $M_1(1)$ portion of dual Selmer, we can try and get auxiliary primes that take care of the remaining part of the group. Indeed, we have the following proposition:
\begin{proposition} Let $\Sc{S} = (\bar{\rho},\mu,S,\{\Lambda_v\}_{v \in S}, \{\Sc{D}_v\}_{v \in S})$ be a global deformation problem. Let $T \subset S$, and suppose for $v \in S-T$ we have $\Sc{D}_v = \Sc{D}_v^{\St(\alpha_v)}$. Suppose further that $h^1_{\Sc{S},T}(M_1(1)) = 0$, and let $N_1 \geq 1$ be an integer. Then there exists a finite set $Q_1$ of finite places of $F$, disjoint from $S$, satisfying:
\begin{itemize}
\item[(1)] We have $\# Q_1 = h^1_{\Sc{S},T}(M_0(1))$, and for each $v \in Q_1$, the norm $q_v \equiv 1 \Mod p^{N_1}$ and $\bar{\rho}(\Frob_v)$ has distinct eigenvalues. 
\item[(2)] Define the augmented deformation problem 
$$\Sc{S}_{Q_1} = (\bar{\rho},\mu,S \cup Q_1,\{\Lambda_v\}_{v \in S} \cup \{\Sc{O}\}_{v \in Q_1},\{\Sc{D}_v\}_{v \in S} \cup \{\Sc{D}_v^{\Box}\}_{v \in Q_1}).$$ Then $h^1_{\Sc{S}_{Q_1},T}(\ad^0 \bar{\rho}(1)) = 0$. 
\end{itemize}
\end{proposition}
\begin{proof} See the proof of Proposition 5.24 in \cite{Thorne}. 
\end{proof}

\section{The Main Theorem}
From this point, everything from sections 6 of \cite{Thorne} carries over. We can now work towards proving the main theorem. We proceed as in the aforementioned paper, making the necessary modifications. First, some preliminary results.

\begin{lemma}\label{BaseChange}  Let $F$ be a totally real number field, and let $F'/F$ be a totally real, soluble extension. Let $p$ be a prime and let $\iota: \overline{\Q}_p \stackrel{\sim}{\to} \C$ be a fixed isomorphism.
\begin{itemize}
\item[(1)] Let $\pi$ be a cuspidal automorphic representaion of $\GL_2(\A_F)$ of weight $2$, and suppose that $r_{\iota}(\pi)|_{G_{F'}}$ is irreducible. Then there exists a cuspidal automorphic representation $\pi_{F'}$ of $\GL_2(\A_{F'})$ of weight $2$, called the \emph{base change} of $\pi$, such that $r_{\iota}(\pi_{F'}) \cong r_{\iota}(\pi)|_{G_{F'}}$. 
\item[(2)] Let $\rho: G_F \to \GL_2(\overline{\Q}_p)$ be a continuous representation such that $\rho|_{G_{F'}}$ is irreducible. Let $\pi'$ be a cuspidal automorphic representation of $\GL_2(\A_{F'})$ of weight $2$ with $\rho|_{G_{F'}} \cong r_{\iota}(\pi')$. Then there exists a cuspidal automorphic representation $\pi$ of $\GL_2(\A_F)$ of weight $2$ such that $\rho \cong r_{\iota}(\pi)$. 
\end{itemize}
\end{lemma}

\begin{proof} This is stated in \cite{Thorne}, Lemma 5.1. The proof follows from results of \cite{Lan}, using arguments of \cite{BLGHT}, Lemma 1.3. 
\end{proof}

\begin{theorem}\label{Long} Let $F$ be a totally real field, and let $p$ be an odd prime. Let $\rho: G_F \to \GL_2(\overline{\Q}_p)$ be a continuous representation. Suppose that:
\begin{itemize}
\item[(1)] $[F:\Q]$ is even.
\item[(2)] Letting $K$ be the quadratic subfield of $F(\zeta_p)/F$, there exists a continuous character $\bar{\chi}: G_K \to \overline{\F}_p^{\times}$ such that $\bar{\rho} \cong \Ind_{G_K}^{G_F} \overline{\chi}$.
\item[(3)] Letting $w \in \Gal(K/F)$ be the nontrivial element, the character $\bar{\gamma} = \bar{\chi}/\bar{\chi}^w$ remains nontrivial even after restriction to $G_{F(\zeta_p)}$ (in particular, $\bar{\rho}$ is irreducible). 
\item[(4)] We have $F(\zeta_p) \not \subset K(\bar{\gamma}\epsilon^{-1}) \cap K(\bar{\gamma}\epsilon)$.
\item[(5)] The character $\psi = \epsilon \det \rho$ is everywhere unramified. 
\item[(6)] The representation $\rho$ is almost everywhere unramified. 
\item[(7)] For each place $v|p$, $\rho|_{G_{F_v}}$ is semi-stable, and $\bar{\rho}|_{G_{F_v}}$ is trivial. For each embedding $\tau: F \hookrightarrow \overline{\Q}_p$, we have $\text{HT}_{\tau}(\rho) = \{0,1\}$. 
\item[(8)] If $v \nmid p$ is a finite place of $F$ at which $\rho$ is ramified, then $q_v \equiv 1\Mod p$, $\text{WD}(\rho|_{G_{F_v}})^{F-ss} \cong \text{rec}_{F_v}^T(\text{St}_2(\chi_v))$ for some unramified character $\chi_v: F_v^{\times} \to \overline{\Q}_p^{\times}$, and $\bar{\rho}|_{G_{F_v}}$ is trivial. The number of such places is even. 
\item[(9)] There exists a cuspidal automorphic representation $\pi$ of $\GL_2(\A_F)$ of weight $2$ and an isomorphism $\iota: \overline{\Q}_p \to \C$ satisfying:
\begin{itemize}
\item[(a)] There is an isomorphism $\overline{r_{\iota}(\pi)} \cong \overline{\rho}$. 
\item[(b)] If $v|p$ and $\rho$ is ordinary, then $\pi_v$ is $\iota$-ordinary and $\pi_v^{U_0(v)} \neq 0$. If $v|p$ and $\rho$ is non-ordinary, then $\pi_v$ is not $\iota$-ordinary and $\pi_v$ is unramified. 
\item[(c)] If $v \nmid p\infty$ and $\rho|_{G_{F_v}}$ is unramified, then $\pi_v$ is unramified. If $v \nmid p\infty$ and $\rho|_{G_{F_v}}$ is ramified, then $\pi_v$ is an unramified twist of the Steinberg representation. 
\end{itemize}
\end{itemize}
Then $\rho$ is automorphic: there exists a cuspidal automorphic representation $\pi'$ of $\GL_2(\A_F)$ of weight $2$ and an isomorphism $\rho \cong r_{\iota}(\pi')$. 
\end{theorem}
\begin{remark} Here, $U_0(v)$ is the set of matrices in $\GL_2(\Sc{O}_{F_v})$ whose reduction modulo a fixed uniformizer of $\Sc{O}_{F_v}$ is upper triangular. 
\end{remark}
\begin{proof} This is Theorem 7.2 \cite{Thorne} with the necessary modifications (namely, the addition of condition (4) instead of the condition that $K$ be totally real). One can just repeat the proof the author gives in that paper, replacing Proposition 5.20 of loc. cit. with Proposition 2.11 from the previous section. 
\end{proof}

Using this theorem, we arrive at the main theorem. 

\begin{theorem}\label{Aut} Let $F$ be a totally real number field, let $p$ be an odd prime, and let $\rho: G_F \to \GL_2(\overline{\Q}_p)$ be a continuous representation satisfying the following:
\begin{itemize}
\item[(1)] The representation $\rho$ is almost everywhere unramified. 
\item[(2)] For each $v|p$ of $F$, the local representation $\rho|_{G_{F_v}}$ is de Rham. For each embedding $\tau: F \hookrightarrow \overline{\Q}_p$, we have $\text{HT}_{\tau}(\rho) = \{0,1\}$. 
\item[(3)] For each complex conjugation $c \in G_F$, we have $\det \rho(c) = -1$. 
\item[(4)] The residual representation $\bar{\rho}$ is absolutely irreducible, but $\bar{\rho}|_{G_{F(\zeta_p)}}$ is a direct sum of two distinct characters. Further suppose that if $K$ is the unique quadratic subfield of $F(\zeta_p)/F$ and $\bar{\gamma}: G_K \to k^{\times}$ is the ratio of the two characters, then we have $F(\zeta_p) \not \subset K(\bar{\gamma}\epsilon^{-1}) \cap K(\bar{\gamma}\epsilon)$. 
\end{itemize}
Then $\rho$ is automorphic: there exists a cuspidal automorphic representation $\pi$ of $\GL_2(\A_F)$ of weight $2$, an isomorphism $\iota: \overline{\Q}_p \to \C$, and an isomorphism $\rho \cong r_{\iota}(\pi)$. 
\end{theorem}
\begin{proof} The proof is exactly the same as Theorem 7.4 of \cite{Thorne}, replacing Theorem 7.2 of loc. cit. with Theorem \ref{Long} above. The idea is to construct a soluble extension $F'/F$ such that $\bar{\rho}|_{G_{F'}}$ satisfies the conditions of Theorem \ref{Long} above. We then apply Lemma \ref{BaseChange} to deduce the automorphy of $\rho$. We should note that in \cite{Thorne}, the author makes use of Corollary 7.4 in loc. cit., but that goes unchanged for us because that corollary made no assumptions on the quadratic subfield $K$. 
\end{proof}

\section{Application to Elliptic Curves}

We can apply this theorem to elliptic curves. In a paper of Freitas, Le Hung, and Siksek (see \cite{FLHS}) the authors prove there are only finitely many non-automorphic elliptic curves over any given totally real field. We can use the theorem above to prove automorphy of  elliptic curves  $E$ defined over totally real fields  $F$ in some cases not covered in the existing literature. This comes about by applying our theorem to the 7-adic representations arising from $E$. Note that  the assumptions of \cite{Thorne} imply that the theorem in that paper can be applied only to $p$-adic representations of elliptic curves with $p$ congruent to 1 modulo 4.  \\ \\ The following theorem is \cite{FLHS}, Theorems 3 and 4.
\begin{theorem}\label{AbsIrr} Let $p \in \{3,5,7\}$. Let $E$ be an elliptic field over a totally real field $F$ and let $\bar{\rho}_{E,p}: G_F \to \GL_2(\F_p)$ be the representation given by the action on the $p$-torsion of $E$. If $\bar{\rho}_{E,p}(G_{F(\zeta_p)})$ is absolutely irreducible, then $E$ is modular. 
\end{theorem}

We will call an elliptic curve $p$-bad if $E[p]$ is an absolutely reducible $\F_p[G_{F(\zeta_p)}]$-module. Otherwise $E$ is $p$-good. The theorem tells us the only elliptic curves $E$ which are potentially non-modular are those which are $p$-bad for $p = 3,5$, and $7$. In \cite{Thorne}, the author deals with some of these remaining cases:
\begin{theorem} Let $E$ be an elliptic curve over a totally real field $F$. Suppose:
\begin{itemize}
\item[(1)] $5$ is not a square in $F$. 
\item[(2)] $E$ has no $F$-rational $5$-isogeny.
\end{itemize}
Then $E$ is modular. 
\end{theorem}

We can prove a similar theorem, but before doing so we recall \cite{FLHS}, Proposition 9.1.
\begin{proposition} Let $F$ be a totally real number field and let $E$ be an elliptic curve over $F$. Suppose $F \cap \Q(\zeta_7) = \Q$ and write $\bar{\rho} = \bar{\rho}_{E,7}$. Suppose $\bar{\rho}$ is irreducible but $\bar{\rho}(G_{F(\zeta_7)})$ is absolutely reducible. Then $\bar{\rho}(G_F)$ is conjugate in $\GL_2(\F_7)$ to one of the groups
$$H_1 = \left\langle \MatTwo{3}{0}{0}{5}, \MatTwo{0}{2}{2}{0} \right\rangle, \quad H_2 = \left\langle \MatTwo{0}{5}{3}{0},\MatTwo{5}{0}{3}{2} \right\rangle.$$ The group $H_1$ has order $36$ and is contained as a subgroup of index $2$ in the normalizer of a split Cartan subgroup. The group $H_2$ has order $48$ and is contained as a subgroup of index $2$ in the normalizer of a non-split Cartan subgroup. The images of $H_1$ and $H_2$ in $\PGL_2(\F_7)$ are isomorphic to $D_3 \cong S_3$ and $D_4$, respectively. 
\end{proposition}

\begin{theorem}\label{Main} Let $F$ be a totally real field, and let $E$ be an elliptic curve over $F$. Suppose:
\begin{itemize}
\item[(1)] $F \cap \Q(\zeta_7) = \Q$. 
\item[(2)] $E$ has no $F$-rational $7$-isogeny.
\item[(3)] Either $\bar{\rho}_{E,7}(G_{F(\zeta_7)})$ is absolutely irreducible, or it is reducible and $\bar{\rho}_{E,7}(G_F)$ is conjugate to the group $H_1$ from the previous proposition. 
\end{itemize}
Then $E$ is modular. 
\end{theorem}
\begin{proof} Let $\rho: G_F \to \GL_2(\Q_7)$ be the representation given by the action of $G_F$ on the \'{e}tale cohomology $H^1(E_{\overline{F}},\Z_7)$, after a choice of basis. The goal is to show $\rho$ is automorphic. Hypothesis (2) is equivalent to $\bar{\rho}$ being irreducible, hence absolutely irreducible because of complex conjugation. If $\bar{\rho}|_{G_{F(\zeta_7)}}$ is irreducible, then $\rho$ is automorphic by Theorem \ref{AbsIrr} above. We now deal with the case when this restriction is not irreducible. \\ \\ If $\bar{\rho}|_{G_{F(\zeta_7)}}$ is absolutely reducible, the third hypothesis combined with the previous proposition gives that the projective image of $\bar{\rho}$ in $\PGL_2(\F_7)$ is isomorphic to $D_3$. This implies that $\bar{\rho}|_{G_{F(\zeta_7)}}$ cannot be scalar since $\Gal(F(\zeta_7)/F)$ is cyclic, and therefore cannot surject onto $D_3$. \\ \\ Let $K$ be the quadratic subfield of $F(\zeta_7)/F$, so that $[F(\zeta_7):K] = 3$ by hypothesis (1). Let $\bar{\gamma}: G_K \to \F_7^{\times}$ be the character which gives the ratio of the eigenvalues. We need to examine the subgroup
$$H_1 = \left \langle \MatTwo{3}{0}{0}{5}, \MatTwo{0}{2}{2}{0} \right \rangle \subset \GL_2(\F_7).$$ It is easy to check that these two matrices generate the projective image as well. By simply checking the ratio of eigenvalues of each of the matrices, one can check that the possible values for the image of $\bar{\gamma}$ are elements of $\{1,2,4\}$. Therefore, $[K(\bar{\gamma}):K]=3$ or $1$. However, it cannot be the latter as $\overline{\gamma}$ is nontrivial as a character on $G_K$ by assumption. Therefore $[K(\bar{\gamma}):K] = 3$, so $K(\bar{\gamma}) \cap F(\zeta_7) = K$ or $F(\zeta_7)$. But we know it cannot be $F(\zeta_7)$ since the image of $\bar{\rho}|_{G_{F(\zeta_7)}}$ is non-scalar. Thus, $K(\bar{\gamma})$ is disjoint over $K$ from $F(\zeta_7)$ and $[K(\bar{\gamma}):K] = [F(\zeta_7):K] = 3$. Thus, hypothesis (4) of the main theorem above is satisfied, and the theorem implies $E$ is modular. 
\end{proof}

We end by remarking that we can extend Theorem \ref{Main} to primes other than $p=7$.
\begin{theorem} Let $F$ be a totally real field, and let $E$ be an elliptic curve over $F$. Let $p \geq 7$ be a prime such that $(p-1)/2 = q^n$ for some odd prime $q$ and $n \geq 1$. Suppose:
\begin{itemize}
\item[(1)] $F \cap \Q(\zeta_p) = \Q$. 
\item[(2)] $E$ has no $F$-rational $p$-isogeny.
\item[(3)] $\bar{\rho}_{E,p}(G_{F})$ normalizes a split Cartan subgroup of $\GL_2(\F_p)$.
\end{itemize}
Then $E$ is modular. 
\end{theorem}
\begin{proof} Let $\rho: G_F \to \GL_2(\Q_p)$ be the representation given by the action of $G_F$ on the \'{e}tale cohomology $H^1(E_{\overline{F}},\Z_p)$, after a choice of basis. The goal is to show $\rho$ is automorphic. Hypothesis (2) is equivalent to $\bar{\rho}$ being irreducible, hence absolutely irreducible because of complex conjugation. Hypothesis (3) says $\bar{\rho}_{E,p}(G_F)$ is contained in the normalizer of a split Cartan subgroup. Thus, $\bar{\rho}(G_{F(\zeta_p)})$ cannot be scalar, since $\Gal(F(\zeta_p)/F)$ is cyclic, and hence cannot surject onto a non-cylic group. \\ \\ Let $K$ be the quadratic subfield of $F(\zeta_p)/F$, so that $[F(\zeta_p):K] = q^n = (p-1)/2$ by hypothesis (1). Let $\bar{\gamma}: G_K \to \F_p^{\times}$ be the character which gives the ratio of eigenvalues of $\bar{\rho}|_{G_K}$. We want to examine $[K(\bar{\gamma}):K]$, where $K(\bar{\gamma}) = \bar{F}^{\ker(\bar{\gamma})}$ as always. In particular, we will show that $K(\bar{\gamma}) \cap F(\zeta_p)$ is a field $L$ which satisfies $(\# \epsilon(G_L),\#\bar{\gamma}(G_L)) > 1$, which implies hypothesis (4) of the main theorem. Note that hypothesis (1) implies that, as a character of $G_K$, that $\epsilon$ takes values in $(\F_p^{\times})^2$. \\ \\ Using the fact that $\det \bar{\rho}$ is the mod $p$ cyclotomic character, we find that $\bar{\chi} \bar{\chi}^w = \epsilon$, so that $\bar{\gamma} = \bar{\chi}/\bar{\chi}^w = \bar{\chi}^2 \epsilon^{-1}$, which is a character $G_K \to (\F_p^{\times})^2$. Thus, the order of $\bar{\gamma}$ divides $(p-1)/2 = q^n$, and moreover cannot equal $1$ as $\bar{\gamma}$ is a nontrivial character of $G_K$. Thus, $1<[K(\bar{\gamma}):K] \big| q^n$. Moreover, $[F(\zeta_p):K] = q^n$ by hypothesis (1) of the theorem. Lastly, we know $K(\bar{\gamma}) \not \subseteq F(\zeta_p)$ since $\bar{\gamma}$ is nontrivial upon restriction to $G_{F(\zeta_p)}$, and thus $K(\bar{\gamma}) \cap F(\zeta_p)$ is neither $K(\bar{\gamma})$ nor $F(\zeta_p)$. This intersection is therefore a field $L$ which satisfies $(\#\epsilon(G_L),\# \bar{\gamma}(G_L))>1$ as $q$ divides both quantities. Thus, hypothesis (4) of the main theorem is satisfied, and therefore $E$ is modular. 
\end{proof}
\bibliographystyle{alpha}{}
\bibliography{Master}
\nocite{*}

\end{document}